\documentclass{amsart}
\usepackage{mathptmx, bbm, amscd, amssymb, enumerate, colonequals, mathdots, xcolor, comment, graphicx, lineno, xspace, amsmath}

\usepackage{color}
\usepackage{tikz-cd}
\definecolor{chianti}{rgb}{0.6,0,0}
\definecolor{meretale}{rgb}{0,0,.6}
\definecolor{leaf}{rgb}{0,.35,0}
\usepackage[colorlinks=true, pagebackref, hyperindex, citecolor=meretale, urlcolor=leaf, linkcolor=chianti]{hyperref}
\usepackage[all]{xy}
\usepackage{paralist}
\usepackage{mathrsfs}

\newtheorem{theorem}{Theorem}[section]
\newtheorem{lemma}[theorem]{Lemma}
\newtheorem{corollary}[theorem]{Corollary}
\newtheorem{proposition}[theorem]{Proposition}

\theoremstyle{definition}
\newtheorem{defn}[theorem]{Definition}
\newtheorem{ntn}[theorem]{Notation}
\newtheorem{example}[theorem]{Example}
\newtheorem{remark}[theorem]{Remark}
\numberwithin{equation}{theorem}

\newtheorem*{maintheorem}{Main Theorem}

\def\ara{\operatorname{ara}}

\def\height{\operatorname{ht}}

\def\init{\operatorname{in}}

\def\ker{\operatorname{ker}}

\def\rank{\operatorname{rank}}

\def\Gr{\operatorname{Gr}}

\def\GL{\operatorname{GL}}
\def\SL{\operatorname{SL}}

\DeclareMathOperator{\RI}{RI}

\def\fraka{\mathfrak{a}}

\def\frakm{\mathfrak{m}}

\def\CC{\mathbb{C}}

\def\NN{\mathbb{N}}

\def\QQ{\mathbb{Q}}

\def\ZZ{\mathbb{Z}}

\def\ge{\geqslant}
\def\le{\leqslant}
\def\phi{\varphi}

\def\to{\longrightarrow}
\def\mapsto{\longmapsto}
\def\into{\hookrightarrow}
\def\onto{{\,\,\longrightarrow\hspace{-1.8ex}\rightarrow}\,\,}

\def\mapsfrom{\mathrel{\reflectbox{\ensuremath{\mapsto}}}}

\begin{document}
\title[arithmetic rank of residual intersections]{The arithmetic rank of the residual intersections \\ of a complete intersection ideal}

\author[Batavia]{Manav Batavia}
\address{Department of Mathematics, Purdue University, 150 N University St., West Lafayette, IN~47907, USA}
\email{mbatavia@purdue.edu}

\author[Mohana Sundaram]{Kesavan Mohana Sundaram}
\address{Department of Mathematics, University of Nebraska, Lincoln, NE 68588-0130, USA}
\email{km2@huskers.unl.edu}

\author[Murray]{Taylor Murray}
\address{Department of Mathematics, University of Nebraska, Lincoln, NE 68588-0130, USA}
\email{tmurray11@huskers.unl.edu}

\author[Pandey]{Vaibhav Pandey}
\address{Department of Mathematics, Purdue University, 150 N University St., West Lafayette, IN~47907, USA}
\email{pandey94@purdue.edu}

\begin{abstract}
The arithmetic rank of an ideal in a polynomial ring over an algebraically closed field is the smallest number of equations needed to define its vanishing locus set-theoretically. We determine the arithmetic rank of the generic $m$-residual intersection of an ideal generated by $n$ indeterminates for all $m\geq n$ and in every characteristic. We further give an explicit description of its set-theoretic generators. Our main result provides a sharp upper bound for the arithmetic rank of any residual intersection of a complete intersection ideal in any Noetherian  local ring. In particular, given a complete intersection ideal of height at least two, any of its generic residual intersections---including its generic link---fails to be a set-theoretic complete intersection in characteristic zero.    
\end{abstract}
\maketitle
\tableofcontents

\section{Introduction}

Let $I$ be an ideal in a Noetherian ring $R$. The \textit{arithmetic rank} of $I$ is the smallest number of elements needed to generate it up to radical
\[\ara(I) \colonequals \min\{k : \text{there exist } f_1,\ldots,f_k \text{ in } R \text{ with} \sqrt{f_1,\ldots,f_k}=\sqrt{I}\}.\]
In particular, if $R$ is a polynomial ring over an algebraically closed field, $\ara(I)$ is the optimal number of equations needed to define the variety $V(I)=V(\sqrt{I})$. The simplicity of the definition of arithmetic rank belies the difficulty in computing it, see \cite{CowsikNori, HartshorneCI}. A part of this difficulty is that the set-theoretic generators may bear little resemblance to the generators of the given ideal \cite[Example 8.2]{LSW}. In addition, the arithmetic rank may depend on the characteristic of the ring \cite{Barile95, BarileLyubeznik}.

In this paper, we study the arithmetic rank of the \textit{residual intersections} of a complete intersection ideal. The notion of residual intersections, introduced by Artin and Nagata \cite{ArtinNagata}, is the higher codimension analogue of \textit{links} of ideals \cite{PeskineSzpiroLinkage}. In a different direction, the residual intersections of complete intersections arise as the defining ideals of certain \textit{varieties of complexes}, introduced by Buchsbaum and Eisenbud \cite{Buchsbaum-Eisenbud} (see \cite[Example 5.15]{Huneke-Ulrich85}). We now describe our results, and the objects involved, more precisely. 

We often call an ideal generated by a regular sequence a complete intersection ideal. Let $f_1,\ldots, f_n$ be a regular sequence in a Noetherian local ring $(R, \frakm)$ and $I \colonequals (\underline{f})$. Let $\fraka \subsetneq I$ be an ideal generated by $m$ elements with $m \geq n$ and set $J=\fraka :I$. If $\height(J)\geq m$, then $J$ is called an $m$-residual intersection of the complete intersection ideal $I$.

Since the ideal $\fraka \colonequals (a_1,\ldots,a_m)$ is contained in $I$, there exist elements $t_{ij}$ in $R$ such that
\[\begin{pmatrix}
    a_1\\
    \vdots\\
    a_m
\end{pmatrix} =\begin{pmatrix}
    t_{11} & t_{12} & \cdots & t_{1n}\\
    \vdots & \vdots &  & \vdots\\
    t_{m1} & t_{m2} & \cdots & t_{mn}
    \end{pmatrix}
    \begin{pmatrix}
    f_1\\
    \vdots\\
    f_n
\end{pmatrix} \colonequals T \underline{f}
   .\]
The generators of such residual intersections were calculated by Huneke and Ulrich in \cite[Example 3.4]{HunekeUlrichRI}:
\[J=I_n(T)+(T\underline{f}),\]
where $I_n(T)$ is the ideal generated by the size $n$-minors of the matrix $T$. The task at hand is to bound the arithmetic rank $\ara(J)$ for \textit{any} such choice of $T$. In order to do this, we work with generic residual intersections.

\begin{defn}
Let $\underline{f}\colonequals f_1,\ldots,f_n$ be a regular sequence in a Noetherian local ring $(R,\frakm)$ and $I \colonequals (\underline{f})$. Let $X$ be an $m \times n$ matrix of indeterminates with $m \geq n$. Let $\mathfrak{a}$ be the ideal generated by the entries of the matrix $X[f_1\dots f_n]^T$. We set 
\[ \RI(m,I)=\RI(m,\underline{f})\colonequals \mathfrak{a}R[X]:IR[X].\] The ideal $\RI(m,I)$ is called the \emph{generic $m$-residual intersection} of $I$. It is independent of the choice of the generating set of $I$ \cite[Lemma 2.2]{HunekeUlrichGenericRI}. 
\end{defn}

Furthermore, we replace the regular sequence $\underline{f}$ by indeterminates $\underline{y}\colonequals y_1,\ldots, y_n$ and focus on the residual intersection ideal $\RI(m,\underline{y})$. Once the arithmetic rank of $\RI(m,\underline{y})$ is determined, we get a sharp upper bound on the arithmetic rank of \textit{any} residual intersection of a complete intersection ideal using \cite[Example 3.4]{HunekeUlrichRI}, since the arithmetic rank can only go down under specialization. 

Indeed, we determine the arithmetic rank of the generic residual intersection $\RI(m,\underline{y})$ for all $m \geq n$ and in every characteristic. 

\begin{maintheorem}(Theorem \ref{thm:arageneric})
Let $m$ and $n$ be positive integers with $m \geq n$ and let $K$ be any field or the integers. Let $X$ and $\underline{y}$ be $m \times n$ and $n \times 1$ matrices of indeterminates respectively over $K$. The arithmetic rank of the generic $m$-residual intersection ideal $\RI(m,\underline{y}) = I_n(X)+(X\underline{y})$ in the polynomial ring $K[X,\underline{y}]$ is
\[
\mathrm{ara}(\RI(m,\underline{y})) = 
\begin{cases} 
n(m-n+1)+1 &   n \ge 2, \\
m & n=1.
\end{cases}
\]
\end{maintheorem}

As indicated above, an important consequence of this result is that it immediately gives a sharp upper bound for the arithmetic rank of \textit{any} residual intersection of a complete intersection ideal in any Noetherian local ring, see Theorem \ref{thm:main}. The bound obtained is much stronger than that in \cite[Corollary 4.3]{HamidJLMS}, where an upper bound on arithmetic rank is given for colon ideals in general (which may not be residual intersections).

Furthermore, via a faithfully flat base change, we also show that the arithmetic rank of the generic $m$-residual intersection of any complete intersection ideal in characteristic zero is as asserted in the Main Theorem. In particular, any such ideal is \textit{not} a set-theoretic complete intersection. We now lay out our strategy for computing the arithmetic rank. 

\begin{enumerate}
    \item In Section \ref{section:InvariantTheory}, we compute the arithmetic rank of $\RI(m,\underline{y})$ in characteristic zero. This is accomplished by identifying an intriguing connection with classical invariant theory and using it to deduce a local cohomology obstruction on the arithmetic rank as in \cite[Theorem 1.1]{JPSW}. 

    \item In Section \ref{section:ASL}, we determine the set-theoretic generators of $\RI(m,\underline{y})$. This is done by endowing the invariant ring of Section \ref{section:InvariantTheory} with the structure of an Algebra with a Straightening Law (ASL) in any characteristic.

    Given the important role it plays in the proof of our Main Theorem, we examine the above invariant/auxiliary ring more closely in Section \ref{section:ASL}. We show that it is a Gorenstein factorial domain. In positive characteristic, we prove that it is $F$-regular. The key insight is to initialize the straightening relations of the auxiliary ring to show that its initial subalgebra is also an ASL which, crucially, is normal. Once this is done, the singularity can be deformed back to the auxiliary ring. 
        
    For the interested reader, in Appendix \ref{section:Appendix}, we discuss an alternative approach to bounding the arithmetic rank of $\RI(m,\underline{y})$ from above which circumvents the ASL machinery. We construct a transcendence basis for the field of fractions of the auxiliary ring from first principles. This recovers the characteristic-free upper bound for the arithmetic rank of $\RI(m,\underline{y})$.

    \item In positive characteristic, the local cohomology obstruction to arithmetic rank, mentioned in item (1), vanishes. In view of this, we shift gears in Section \ref{section:Topology} and study the topology of the open complement of the variety of $\RI(m,\underline{y})$. The rough idea is to show that this open set \textit{cannot} be covered by a `small' number of basic affine open sets. Indeed, we show that singular cohomology recovers the lower bound enforced by local cohomology in characteristic zero. Importantly, with some effort, the topological arguments extend analogously to positive characteristic as well to give an \'etale cohomology obstruction, thereby finishing the proof. 
\end{enumerate}

Observe that since the arithmetic rank is subadditive, we a priori have the upper bound
\[\ara(\RI(m,\underline{y})) \leq \ara(I_n(X))+\ara(X\underline{y})=(mn-n^2+1)+m\]
using \cite{BrunsSchwanzl}. Our main result is that, independent of characteristic,
\[\ara(\RI(m,\underline{y}))=mn-n^2+1+n,\]
which is better than the above upper bound exactly by the difference $m-n$ of the codimensions of the residual intersection $\RI(m,\underline{y})$ and the complete intersection ideal $(\underline{y})$.


\section{A curious connection with classical invariant theory} \label{section:InvariantTheory}

The aim of this section is to compute the arithmetic rank of the generic $m$-residual intersection $\RI(m,\underline{y})$ of the complete intersection ideal $(\underline{y})\colonequals (y_1,\ldots,y_n)$ in characteristic zero for every $m \geq n$. This will be done by identifying an amusing connection with classical invariant theory.

Let $m\geq n$ be positive integers. Let $X\colonequals (x_{ij})$ and $\underline{y}\colonequals (y_i)$ be $m \times n$ and $n \times 1$ matrices of indeterminates respectively over a field $K$. Consider the linear action of the special linear group $\SL_n(K)$ on the polynomial ring $S \colonequals K[X,\underline{y}]$ as follows:
\[
M\colon\begin{cases} X & \mapsto XM^{-1}\\ \underline{y} & \mapsto M\underline{y}.\end{cases}
\]
Observe that this action is given by the regular representation of $\SL_n$ and by copies of its dual. Note that each element of the set
\[B \colonequals \{ \text{all entries of }\; X\underline{y} \; \text{and all size }n\text{-minors of } \; X  \}\]
is fixed under this action. When $K$ is infinite, the \textit{First Fundamental Theorem} for $\SL_n(K)$ states that the invariant ring is precisely the $K$-algebra generated by this set, i.e., $S^{\SL_n(K)}=K[B]$, see \cite[Theorem 3.3]{DeConciniProcesi76} or \cite[Theorem 4.4.4]{DerksenKemper}.

The \textit{key observation} of this section is that when the field $K$ is infinite, the generic $m$-residual intersection $\RI(m,\underline{y})$ arises as the \emph{nullcone ideal} $\frakm_{S^{\SL_n(K)}}S$ of the above action, where $\frakm_{S^{\SL_n(K)}}$ is the homogeneous maximal ideal of the invariant ring $S^{\SL_n(K)}$. In other words,
\[\RI(m,\underline{y})=\frakm_{S^{\SL_n(K)}}S\]
is the ideal generated by of all positive degree invariants in $S$. Having made this observation, the computation of the arithmetic rank follows rather elegantly from \cite[Theorem 1.1]{JPSW}. We write down the proof for the convenience of the reader.

\begin{proposition}\label{thm:nullcone}
Let $K$ be a field of characteristic zero. Then
\[\ara(\RI(m,\underline{y}))=\dim S^{\SL_n(K)}.\]
\end{proposition}

\begin{proof}
 Set $T \colonequals S^{\SL_n(K)}$ and $d \colonequals \dim T$. Any homogeneous system of parameters of the invariant ring $T$ generates its homogeneous maximal ideal $\frakm_{T}$ up to radical. Viewing this in the polynomial ring $S$, we get
\[\ara(\RI(m,\underline{y})) \leq d,\]
independent of characteristic.

Since the special linear group is linearly reductive in characteristic zero, the inclusion
$T \into S$ of $T$-modules splits. It follows that the local cohomology module supported at the nullcone ideal
\begin{equation}
\label{equation:theorem:ara:intro}
H^d_{\frakm_T}(T)\otimes_TS\ =\ H^d_{\frakm_T}(S)\ =\ H^d_{\frakm_TS}(S)=\ H^d_{\RI(m,\underline{y})}(S)
\end{equation}
is nonzero. But then $\ara(\RI(m,\underline{y}))\ge d$ since the local cohomology $H^d_{\RI(m,\underline{y})}(S)$ can be computed by the \v{C}ech complex on the $\ara(\RI(m,\underline{y}))$ many generators witnessing the arithmetic rank of $\RI(m,\underline{y})$ (see~\cite[Proposition~9.12]{24hours}).
\end{proof}

While the connection with invariant theory is remarkably effective in characteristic zero, the `local cohomology obstruction' in the above proof vanishes in positive characteristic. Indeed, we show next the stronger statement that the above inclusion does \textit{not} split in positive characteristic; compare with \cite[Theorem 1.1]{HJPS}. In order to do this, we would need the dimension of the `auxiliary ring' $K[B]$, which is calculated in the next section (see Theorem \ref{thm:ASLdimension}) by endowing it with the structure of an Algebra with a Straightening Law.

\begin{lemma} \label{lem:PS}
Let $K$ be a field of positive characteristic. The inclusion
\[K[B]\into K[X,\underline{y}]\]
has no $K[B]$-linear splitting for any choice of $m$ and $n$ with $m \geq n$.
\end{lemma}

\begin{proof}
Set $S\colonequals K[X,\underline{y}]$. By Theorem \ref{thm:ASLdimension}, the ring $K[B]$ has dimension $n(m-n+1)+1$, Assume that the given inclusion splits. Then by the argument in Proposition \ref{thm:nullcone}, we must have that the local cohomology module
 \[H^{n(m-n+1)+1}_{\RI(m,\underline{y})}(S)\]
is nonzero. However, the coordinate ring of the nullcone
\[S/\RI(m,\underline{y}) = S/I_n(X)+(X\underline{y})\]
is a Cohen--Macaulay domain since it is a variety of complexes, see \cite[Theorem 6.2]{HunekeVoC}. By the vanishing theorem of Peskine--Szpiro \cite[Proposition~III.4.1]{PS74}, we then have
\[H^i_{\RI(m,\underline{y})}(S) \neq 0 \quad \text{if and only if } i=\height(\RI(m,\underline{y})).\]
The $m$-residual intersection $\RI(m,\underline{y})$ has height $m$. This forces
\[n(m-n+1)+1=m.\]
Rearranging terms, we get
\[(n-1)(m-n)=-1.\]
This is a contradiction since $m \geq n$.
\end{proof}

Another consequence of the invariant ring being an ASL is that the local cohomology obstruction yields the arithmetic rank over the integers as well.

\begin{remark}\label{rmk:Z}
Replacing the field $K$ by the integers $\ZZ$, we have the inclusion
\[\ZZ[B] \into \ZZ[X,\underline{y}].\]
Note that the argument in the beginning of Proposition \ref{thm:nullcone} also applies on replacing the field by the integers since $\ZZ[B]$ is an ASL by Theorem \ref{thm:ASL}. The point is that the ASL structure guarantees a ``$\ZZ$-homogeneous" system of parameters of the ring $\ZZ[B]$ by \cite[Proposition 5.10]{BrunsVetter}. Furthermore, the local cohomology obstruction clearly persists over the integers. Therefore, the arithmetic rank of the residual intersection ideal $(B)\ZZ[X,\underline{y}]$ is the same number as asserted in Proposition \ref{thm:nullcone}. 
\end{remark}

\section{Upper bound on the arithmetic rank: Structure of the auxiliary ring} \label{section:ASL}

Note that Proposition \ref{thm:nullcone} shows that \textit{any} choice of a homogeneous system of parameters of the invariant ring witnesses the arithmetic rank of $\RI(m,\underline{y})$. In this section, we make these polynomials explicit. This is done by finding an ASL structure on the invariant ring. 

We remark that the generators of the invariant ring $S^{\SL_n(K)}$ from the previous section are listed in \cite[Theorem 3.3]{DeConciniProcesi76} and the relations between these generators can be found in \cite[\S 9.4]{PopovVineberg}. We were, however, unable to find a reference for the Krull dimension of the invariant ring in the literature. We are grateful to De Concini and Procesi for their detailed feedback on the content of this section.   

We continue with the notation established in the previous section: $m \geq n$ are positive integers. $X\colonequals (x_{ij})$ and $\underline{y} \colonequals (y_i)$ are $m\times n$ and $n \times 1$ matrices of indeterminates respectively. We set $Q_i$ to be the $i$-th entry of the matrix $X\underline{y}$ and $[i_1,\dots,i_n]$ to be the size $n$-minor of $X$ with rows $i_1<i_2<\cdots <i_n$. Let 
\[B \colonequals \{\{Q_1,\dots,Q_m\}\cup\{[i_1,\dots,i_n]|1\leq i_1<\cdots<i_n\leq m\}\] be the set of $K$-algebra generators of $K[B]$, where $K$ is any ring. Since the main arguments of this section work over any base ring, we drop the usage of `invariant ring' and instead use `auxiliary ring' when referring to $K[B]$ from now on.

\subsection{The auxiliary ring is an ASL}

We define a partial order $<$ on $B$ as follows:
\begin{itemize}
    \item $Q_i \leq Q_j$ if $i \leq j$.
    \item $Q_j \leq [i_1 , i_2 , \ldots , i_n]$ if $j \leq i_n$.
    \item $[i_1,\dots,i_n]\leq [j_1,\dots,j_n]$ if $i_k\leq j_k$ for all $k$.
\end{itemize}

Let us explain this partial order with an example.

\begin{example}\label{example:posetorder}
    Consider the case where $X$ and $\underline{y}$ are respectively $4 \times 2$ and $2 \times 1$ matrices of indeterminates. Then, with the convention that smaller elements are at the top, the partial order on $B$ is given by the following Hasse diagram:
   
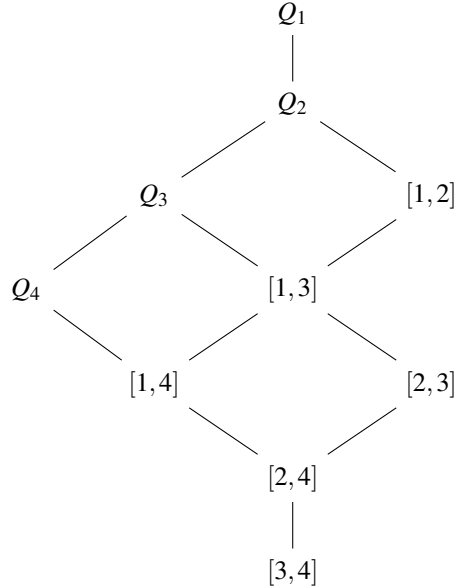
\[\begin{tikzcd}
	&& {Q_1} \\
	&& {Q_2} \\
	& {Q_3} && {[1,2]} \\
	{Q_4} && {[1,3]} \\
	& {[1,4]} && {[2,3]} \\
	&& {[2,4]} \\
	&& {[3,4]}
	\arrow[no head, from=1-3, to=2-3]
	\arrow[no head, from=2-3, to=3-2]
	\arrow[no head, from=2-3, to=3-4]
	\arrow[no head, from=3-2, to=4-1]
	\arrow[no head, from=3-2, to=4-3]
	\arrow[no head, from=3-4, to=4-3]
	\arrow[no head, from=4-1, to=5-2]
	\arrow[no head, from=4-3, to=5-2]
	\arrow[no head, from=4-3, to=5-4]
	\arrow[no head, from=5-2, to=6-3]
	\arrow[no head, from=5-4, to=6-3]
	\arrow[no head, from=6-3, to=7-3]
\end{tikzcd}\]
\\
\end{example}

We recall the definition of an ASL.

\begin{defn}
    Suppose $A$ is a ring and that $H \subset A$ is a finite poset. A \textit{standard monomial} is a product of a totally ordered set of elements of $H$: $$ \alpha_1 \cdots \alpha_k \ \text{where} \ \alpha_{1} \leq \alpha_{2} \leq \cdots \leq \alpha_{k}. $$
\end{defn}

\noindent Assume now that $A$ is a $K$-algebra for a ring $K$ and that the elements of $H$ generate $A$ as a $K$-algebra. 
\begin{defn}
    The ring $A=K[H]$ is an \textit{Algebra with a Straightening Law} if it satisfies the following axioms:
\begin{itemize}
    \item (ASL-1) The algebra $A$ is a free $K$-module with basis given by the set of standard monomials.
    \item (ASL-2) If $\alpha, \beta \in H$ are incomparable, and if 
    $$\alpha \cdot \beta = \sum_i c_i \cdot (\gamma_{i1} \cdots \gamma_{i2} \cdots \gamma_{in})$$
    is the unique expression of $\alpha \cdot \beta$ as a linear combination of standard monomials (here $c_i$ is nonzero and $\gamma_{i1}\leq \cdots \leq \gamma_{in}$ for all $i$) , then $\gamma_{i1} < \alpha$ and $\gamma_{i1} < \beta$ for all $i$. 
    The relations in (ASL-2) are called the\textit{ straightening relations} of the algebra $A$. 
\end{itemize}
\end{defn}


\begin{theorem}\label{thm:ASL}
    Let $K$ be any ring. The auxiliary ring $K[B]$ is an ASL. 
\end{theorem}

\begin{proof}
    We first prove that the standard monomials in $B$ are linearly independent over $K$. Consider the lexicographic monomial order defined by 
    \[y_1<y_2<\cdots<y_n<x_{11}<x_{12}<\cdots<x_{1n}<x_{21}<\cdots<x_{mn}.\] 
    Under this monomial order, the initial monomials of $Q_i$ and $[j_1,\dots,j_n]$ in $K[X,\underline{y}]$ are: 
    \[\init(Q_i)=x_{i,n}y_n \; \text{ and } \; \init([j_1,\dots,j_n])=x_{j_1,1}\cdots x_{j_n,n}.\]
    Note that any standard monomial in $B$ can be expressed as $Q\mu$, where $Q$ is a standard monomial in the $Q_i$ and $\mu$ is a standard monomial in the maximal minors. Given any $Q_i$ with nonzero exponent in $Q$ and any minor $[j_1,\dots,j_n]$ with nonzero exponent in $\mu$, we have $i\leq j_n.$ To prove the $K$-linear independence of the standard monomials, it suffices to prove that every standard monomial has a distinct initial monomial. Here, it may be prudent to note that standard monomials are monomials in $K[B]$, but polynomials in $K[X,\underline{y}]$.
    
    Suppose that two distinct standard monomials $Q\mu$ and $Q'\mu'$ have the same initial monomial, i.e., \[\init(Q)\init(\mu) = \init(Q')\init(\mu').\] If $Q=Q'$ then $\init(\mu)=\init(\mu')$; but then $\mu=\mu'$ by \cite[Proposition 3.3.4]{BCRV}. Thus we must have $Q\neq Q'$. Let 
    \[Q=Q_{i_1}\cdots Q_{i_t}\; \text{and } \; Q'=Q_{j_1}\cdots Q_{j_t}.\] 
    Since the degree of $y_n$ in $\init(Q)\init(\mu)$ and $\init(Q')\init(\mu')$ is the same, the degree of $Q$ must be equal to the degree of $Q'$ and we can write 
    \[\init(Q) = (x_{i_1,n}.y_n)\cdots(x_{i_t,n}.y_n) \; \text{and} \; \init(Q') = (x_{j_1,n}.y_n)\cdots(x_{j_t,n}.y_n).\] Since $\init(Q) \neq \init(Q')$, we may assume without loss of generality that $\init(Q)$ has a higher exponent of $x_{in}$ than $\init(Q')$ for some $i$. Then $\init(\mu')$ has a higher exponent of $x_{in}$ than $\init(\mu).$ Consequently, by a degree counting argument, $\init(\mu)$ has a higher exponent of $x_{kn}$ for some $k\neq i$ and $\init(Q)$ has a lower exponent of $x_{kn}$ than $\init(Q')$. As $Q\mu$ is standard, $Q_i<[\ldots k]$, which implies $i<k$. Similarly, as $Q'\mu'$ is standard, we conclude $k<i$, a contradiction. Therefore distinct standard monomials have distinct initial monomials and we are done. The fact that the standard monomials span $K[B]$ as a $K$-vector space follows from (ASL-2) below.
    
    We now prove (ASL-2). We partition $B$ into sets \[\mathcal{C}=\{Q_1,\dots,Q_m\}\; \text{and} \; \mathcal{D}=\{[i_1,\dots,i_n]|1\leq i_1<\cdots < i_n\leq m\}.\] Let $\alpha$ and $\beta$ be incomparable elements in $H$. Since $\mathcal{C}$ is totally ordered, we have $\alpha\in \mathcal{D}$ or $\beta\in\mathcal{D}$. If both $\alpha$ and $\beta$ lie in $\mathcal{D}$, then $\alpha\beta$ can be expressed as required since $K[\mathcal{D}]$ is an ASL with the poset ordering induced by the one on $B$ by \cite[Theorem 4.3]{BrunsVetter}. 

    \noindent Suppose now that $\alpha\in \mathcal{C}$ and $\beta\in\mathcal{D}$. Let $\alpha=Q_j$ and $\beta=[i_1,\dots,i_n]$. As $\alpha$ and $\beta$ are incomparable, we have $j>i_n$. Note that the determinant
    $$\begin{vmatrix}
        x_{i_1,1} & x_{i_1,2} & \cdots & x_{i_1,n} & Q_{i_1}\\
        x_{i_2,1} & x_{i_2,2} & \cdots & x_{i_2,n} & Q_{i_2}\\
        \vdots & \vdots & \ddots & \vdots & \vdots\\     
        x_{i_n,1} & x_{i_n,2} & \cdots & x_{i_n,n} & Q_{i_n}\\
        x_{j,1} & x_{j,2} & \cdots & x_{j,n} & Q_{j}\\
    \end{vmatrix}=0$$ as the last column is a linear combination of the other columns. Expanding the determinant along the last column allows us to express $Q_j[i_1,\dots,i_n]$ as a $K$-linear combination of the monomials $Q_{i_t}[i_1,\dots,\hat{i_t},\dots,i_n,j]$ for $\ 1\leq t \leq n$. Each of these monomials is standard as $i_t\leq i_n < j.$ We also have $Q_{i_t}<\alpha$ and $Q_{i_t}<\beta$, as required. 
\end{proof}

Now that we know $K[B]$ is an ASL, there is a canonical way to obtain its homogeneous system of parameters. We recall this construction next:

\begin{defn}[{\cite[Page 55]{BrunsVetter}}]
   Given a finite poset $H$ and $\mu\in H$, define $$\text{rank}(\mu)=\max\{k|\text{ there exists a chain of }\mu=\mu_k>\mu_{k-1}>\cdots>\mu_1 \text{ for }\mu_i\in H\}.$$ We also define $\text{rank}(H)=\max\{\text{rank}(\mu):\mu\in H\}.$ 
\end{defn}

\begin{proposition}[{\cite[Lemma 5.9]{BrunsVetter}}]\label{prop:SOPofASL}
Let $H$ be a finite poset and $A=K[H]$ be an ASL. Let $x_i=\sum_{\text{rank}(\mu)=i}\mu$ for all $1\leq i\leq \rank(H)$. Then, $\{x_1,\dots,x_{rank(H)}\}$ is a system of parameters of $A$.    
\end{proposition}

With these combinatorial inputs at our disposal, we are able to make explicit the polynomials which witness the arithmetic rank of $\RI(m,\underline{y})$.   

\begin{theorem}\label{thm:ASLdimension}
Let $K$ be a field. The ring $K[B]$ has dimension $n(m-n+1)+1$ with a homogeneous system of parameters given by 
    \[
\left\{ \sum_{\substack{\operatorname{rank}(\mu)=i}} \mu \;:\; i=1,\ldots,n(m-n+1)+1 \right\}.
\]

\end{theorem}
\begin{proof}
    It may be observed that the rank of the poset $B$ is $n(m-n+1)+1$. One way to see this is by noting the existence of the chain
    \begin{align*}
    Q_1<\cdots<Q_n<[1,2,\ldots n]<[1,\ldots,n-1,n+1]<\cdots<[1,\ldots,n-1,m]<\\
    [1,\ldots,n-2,n,m]<[1,\ldots,n-2,n+1,m]<\cdots<[m-n+1,\ldots,m];
    \end{align*}
    the rest follows from Proposition \ref{prop:SOPofASL}.
\end{proof}

\begin{example}
    Returning to Example \ref{example:posetorder}, we have $\dim(K[B])=7$; the polynomials \[Q_1,\; Q_2,\; Q_3+[1,2],\; Q_4+[1,3], \;[1,4]+[2,3],\; [2,4], \; \text{and }[3,4]\] form a system of parameters of $K[B]$. Equivalently, they realize the arithmetic rank of the ideal $\RI(4,(y_1,y_2))$.
\end{example}

\subsection{The auxiliary ring is a Gorenstein factorial domain}

In this subsection, we study the homological properties of the auxiliary ring $K[B]$, where $K$ is a field; in the next subsection we study its singularity in positive characteristic. We clarify that these results have no bearing on the arithmetic rank of the residual intersection under consideration.

We first show that the auxiliary ring is Cohen--Macaulay. This will be achieved by studying the underlying combinatorial structure more closely.

\begin{defn}
    An element $\beta \in H$ is a \textit{cover} of $\alpha \in H$ if $\beta > \alpha$ and if there is no element $\gamma$ lying strictly between $\alpha$ and $\beta$, satisfying $\beta > \gamma >\alpha$. 
    
    The poset $H$ is said to be \textit{wonderful} if the following holds after a smallest and a greatest element $-\infty$ and $\infty$, respectively, have been added to $H$: 
    
    If $\alpha \in H \ \cup \ \{ -\infty \}$, $\gamma \in H \ \cup \ \{ \infty \}$, and $\beta_1,\beta_2 \in H$ are covers of $\alpha$ satisfying $\beta_1<\gamma$ and $\beta_2 < \gamma$, then there exists an element $\beta \in H \ \cup \ \{ \infty \}$ with $\beta \leq \gamma$ which covers both $\beta_1$ and $\beta_2$. 
\end{defn}

Since a graded ASL on a wonderful poset is Cohen--Macaulay by \cite[Theorem 5.14]{BrunsVetter}, we have 

\begin{lemma}\label{lem:wondeful}
    $K[B]$ is a Cohen--Macaulay domain.
\end{lemma}
\begin{proof}
    Indeed, we show that the poset $B$ is wonderful. Let $\alpha\in B\cup\{-\infty\}$. We know that the poset $\{[i_1,\dots,i_n]|1\leq i_1\leq\cdots i_n\leq m\}$ with ordering induced by the ordering on $B$ is wonderful by \cite[Corollary 5.17]{BrunsVetter}. So we may assume that $\alpha\in\{Q_1,\dots,Q_m\}\cup\{-\infty\}.$ If $\alpha=-\infty$, the only cover of $\alpha$ is $Q_1$. Suppose that $\alpha=Q_i$. If $i<n$, the only cover of $\alpha$ is $Q_{i+1}$. On the other hand, if $i\geq n$, the covers of $\alpha$ are $Q_{i+1}$ and $[1,2,\dots,n-1, i]$. Thus, if $i=m$ or $i<n$, $\alpha$ has a unique cover, and we are done. 

    Let $\alpha=Q_i$ for $n\leq i<m$. Let $\gamma\in B\cup\{\infty\}$ such that $Q_{i+1}$ and $[1,\dots,n-1,i]$ are less than $\gamma$. Observe that $\beta:=[1,\dots,n-1,i+1]$ covers both $Q_{i+1}$ and $[1,\dots,n-1,i]$. If $\gamma=\infty$, there is nothing to prove. As $[1,\dots,n-1,i]<\gamma$, $\gamma$ cannot be one of the $Q_i$. Suppose $\gamma=[j_1,\dots,j_n].$ Then, $Q_{i+1}<\gamma$ implies $i+1\leq j_n$. We conclude that 
    \[\beta=[1,\dots,n-1,i+1]\leq\gamma=[j_1,\dots,j_n]. \qedhere\]
\end{proof}

We next show that the auxiliary ring is Gorenstein. The following lemma will be useful to this end. 

\begin{lemma}[{\cite[Page 297]{BrunsHerzog}}]  \label{lem:UFD}
    Let $S$ be a polynomial ring over a field $K$, and $G$ a group acting on $S$ by $K$-algebra automorphisms. If there are no nontrivial homomorphisms from $G$ to $K^{\times}$, then the invariant ring $S^G$ is a unique factorization domain (UFD).
\end{lemma}

\begin{theorem}\label{thm:Gorenstein}
The auxiliary ring $K[B]$ is Gorenstein. 
\end{theorem}

\begin{proof}
    If the field $K$ is infinite, then $K[B]$ arises as the invariant ring of the action of the special linear group $\SL_n(K)$, as discussed in the beginning of Section \ref{section:InvariantTheory}. Since the commutator subgroup of $\SL_n(K)$ is itself, any group homomorphism from $\SL_n(K)$ to the Abelian group $K^{\times}$ is trivial. It follows from Lemma \ref{lem:UFD} that $K[B]$ is a UFD; it is Cohen--Macaulay by Lemma \ref{lem:wondeful}. The assertion now follows from the fact that a finitely generated Cohen--Macaulay algebra over a field which is UFD is also Gorenstein \cite{MurthyUFD}; compare with \cite[Remark 3.10]{SinghGIAN}. 

    If the field $K$ is finite, let $\overline{K}$ be its algebraic closure. The inclusion
    \[K[B] \into \overline{K}[B]\]
    induces a natural surjection on the divisor class groups. The assertion now follows from the above paragraph.
\end{proof}

\subsection{The auxiliary ring is F-regular}

In this subsection, we prove that when the base field $K$ has positive characteristic, the auxiliary ring $K[B]$ is $F$-regular. The key observation is to initialize the straightening relations on the auxiliary ring to deduce that its initial subalgebra is a normal ASL. We then deform the singularity to the auxiliary ring.

We first recall the monomial order defined in the proof of Theorem \ref{thm:ASL}: the lexicographic monomial order defined by 
    \[y_1<y_2<\cdots<y_n<x_{11}<x_{12}<\cdots<x_{1n}<x_{21}<\cdots<x_{mn}.\] 
    Under this monomial order, the initial monomials of $Q_i$ and $[j_1,\dots,j_n]$ are: 
    \[\init(Q_i)=x_{i,n}y_n \; \text{ and } \; \init([j_1,\dots,j_n])=x_{j_1,1}\cdots x_{j_n,n}.\]
\begin{lemma}\label{lemma:sagbi}
    The initial subalgebra of $K[B]$ with respect to the aforementioned monomial order is $K[\init(f)\;|\;f\in B]$. In other words, $B$ is a Sagbi basis of $K[B]$.
\end{lemma}
\begin{proof}
    Let $F\in K[B].$ By Theorem \ref{thm:ASL} (ASL-1), $F$ can be expressed uniquely as a $K$-linear combination of standard monomials. However, as seen in the proof of Theorem \ref{thm:ASL}, every standard monomial has a distinct initial monomial. Consequently, $\init(F)=\init(g)$ for some standard monomial $g$. As $g$ is a product of elements in $B$, $\init(g)\in K[\init(f)\;|\;f\in B].$  
\end{proof}

Let $\init(B) = \{\init(f)\;|\;f \in B\}$. Imitating our partial order on $B$, we define a partial order on $\init(B)$ as follows:
\begin{itemize}
    \item $\init(Q_i) \leq \init(Q_j)$ if $i \leq j$.
    \item $\init(Q_j) \leq \init([i_1 , i_2 , \ldots , i_n])$ if $j \leq i_n$.
    \item $\init([i_1,\dots,i_n])\leq \init([j_1,\dots,j_n])$ if $i_k\leq j_k$ for all $k$.
\end{itemize} 

The key step in establishing the $F$-regularity is:

\begin{lemma}
    The initial subalgebra $\init (K[B])=K[\init(B)]$ is an ASL.
\end{lemma}
\begin{proof}
    We begin by proving (ASL-1). The standard monomials in $K[\init(B)]$ are the initial monomials of the standard monomials in the ASL $K[B]$. As every standard monomial in $K[B]$ has a distinct initial monomial, we infer that the standard monomials in $K[\init(B)]$ are $K$-linearly independent. 
    
    To prove that the standard monomials form a $K$-spanning set of $K[\init(B)]$, let $q$ be a monomial in $K[\init(B)].$ Then, $q=\init(F)$ for some $F\in K[B]$, and as observed in the proof of Lemma \ref{lemma:sagbi}, $\init(F)=\init(g)$ for some standard monomial $g$ in $K[B].$ Since $\init(g)$ is a standard monomial in $K[\init(B)]$, we have proved (ASL-1).
    
    To prove (ASL-2), let $\alpha$ and $\beta$ be two incomparable elements in $K[\init(B)].$ Then, $\alpha=\init(\alpha')$ and $\beta=\init(\beta')$ for incomparable $\alpha'$ and $\beta'$ in $B$. By initializing the straightening relation on $\alpha'\beta'$, we get a corresponding straightening relation on the product $\alpha\beta.$
\end{proof}
\begin{theorem}\label{thm:initialnormal}
    The initial subalgebra $\init(K[B])=K[\init(B)]$ is normal.
\end{theorem}
\begin{proof}
    Consider a minimal presentation of the initial subalgebra $K[\init(B)]$ by the polynomial ring $K[Y_i|1\leq i\leq m+\binom{m}{n}]$; that is, $\phi : K[\underline{Y}] \onto K[\init(B)]$. As $K[\init(B)]$ is an ASL, the presentation ideal $\ker(\phi)$ is generated by the straightening relations of $K[\init(B)].$ We claim that there exists a monomial order $\tau$ in $K[\underline{Y}]$ such that $\init_\tau(\ker(\phi))$ is squarefree. Then $K[\init(B)]$ is normal by \cite[Theorem 6.1.13]{BCRV}.
    
    We first consider the partial order induced on the $Y_i$: $Y_i<Y_j$ if $\phi(Y_i)<\phi(Y_j).$ We extend this partial order to a total order on the $Y_i$ and let $\tau$ be the induced graded reverse lexicographic monomial order in $K[\underline{Y}].$ 
    
    Observe that $\ker(\phi)$ is generated by quadratic binomials $Y_iY_j-Y_sY_t,$ where $Y_i$ and $Y_j$ are incomparable, $Y_s\leq Y_t$, $Y_s<Y_i$ and $Y_s<Y_j.$ The initial monomial of this binomial is \[\init_\tau(Y_iY_j-Y_sY_t)=Y_iY_j.\] Thus, every incomparable pair of $\init(B)$ uniquely corresponds to the initial monomial of a generator of $\ker(\phi).$ As the standard monomials in $K[\init(B)]$ are linearly independent, the quadratic binomials $Y_iY_j-Y_sY_t$ induced by the straightening relations form a Gr\"obner basis of $\ker(\phi)$ with respect to $\tau$. The initial terms of these binomials are squarefree. 
\end{proof}

\begin{corollary}\label{cor:F-regular}
Let $K$ be an $F$-finite field of positive characteristic. $K[B]$ is a strongly $F$-regular ring.
\end{corollary}
\begin{proof}
    A normal affine semigroup ring is a direct summand of a polynomial ring \cite{Hoc72} and therefore it is $F$-rational. Hence by Theorem \ref{thm:initialnormal}, the initial subalgebra $\init(K[B])$ is $F$-rational. As $F$-rationality deforms, we have that $K[B]$ is $F$-rational \cite[Theorem 7.3.11]{BCRV}. Since $K[B]$ is an $\NN$-graded Gorenstein ring by Theorem \ref{thm:Gorenstein}, $F$-rationality and strong $F$-regularity are equivalent. 
\end{proof}

\section{Lower bound on the arithmetic rank: Topology of the residual intersection}\label{section:Topology}

The `algebraic obstruction' which yields the lower bound for the arithmetic rank in Proposition \ref{thm:nullcone} vanishes in positive characteristic as noted in Lemma \ref{lem:PS}. In this section, we establish the lower bound by studying the topology of the complement of the variety $V(\RI(m,\underline{y}))$. We broadly follow the general strategy of Bruns and Schw\"anzl as in \cite{BrunsSchwanzl}, though many new intricacies arise in our case. 

We begin with setting up the notation. Let $m>n>1$ be positive integers (the cases $m=n$ and $n=1$ will be handled separately). Let $X\colonequals (x_{ij})$ and $\underline{y} \colonequals (y_i)$ be $m\times n$ and $n \times 1$ matrices of indeterminates. The ideal $\RI(m,\underline{y})=I_n(X)+(X\underline{y})$ is the generic $m$-residual intersection of the ideal of variables $(\underline{y})$.
We now begin our topological calculations. Let 
\[U\colonequals (K^{m\times n}\times K^{n\times 1})\setminus V(\RI(m,\underline{y}))=\{(A,B)\in K^{m\times n}\times K^{n\times 1}|AB\neq 0 \text{ or } I_n(A)\neq 0\}.\] 
Set $d \colonequals mn+n$ to be the dimension of $K^{m \times n} \times K^{n \times 1}$. Further, let 
\[U_1 \colonequals (K^{m \times n} \times K^{n \times 1}) \setminus V(X\underline{y}) \; \text{and}\; U_2 \colonequals (K^{m \times n} \times K^{n \times 1}) \setminus  V(I_n(X)),\]
so that $U = U_1 \cup U_2$. Similarly, set 
\[
\begin{aligned}
U_{12} \colonequals U_1 \cap U_2
   &= \{\, (A,B) \in K^{m \times n} \times K^{n \times 1} 
          \mid AB \neq 0,\; \operatorname{rank}(A) = n \,\} \\
   &= \{\, (A,B) \in K^{m \times n} \times K^{n \times 1} 
          \mid B \neq 0,\; \operatorname{rank}(A) = n \,\}.
\end{aligned}
\]

\noindent We define a continuous surjection \[
\begin{array}{c}
\pi : U_{12} \longrightarrow Gr(n,m) \times Gr(1,n) \\[6pt]
(A,B) \longmapsto \bigl(\operatorname{im}(A),\, \operatorname{im}(B)\bigr).
\end{array}
\]
Notice that \[\pi^{-1}(K\langle e_1,\dots,e_n\rangle,K\langle (1,0,\dots,0)\rangle) = \GL_n(K) \times \GL_1(K).\]

A key observation of this section is:

\begin{lemma} \label{lem:LTFB}
    The map $\pi : U_{12} \to \Gr(n,m) \times \Gr(1,n)$ is a locally trivial fiber bundle (in the Zariski topology) with fiber $\GL_n(K) \times \GL_1(K).$
\end{lemma}

\begin{proof}
Fix bases $\{e_1,\dots,e_m\}$ of $K^m$ and $\{f_1,\dots,f_n\}$ of $K^n.$ Let $\mathcal{S}\subset \Gr(n,m)\times\Gr(1,n)$ be the collection of pairs of subspaces $(V_1,V_2)$ such that \[V_1\cap K\langle e_{n+1},\dots,e_m\rangle =0 \quad \text{and} \quad V_2\cap K\langle f_2,\dots,f_n\rangle=0.\] 

\noindent Given such a pair $(V_1,V_2)$, we choose a basis $\{\alpha_1,\dots,\alpha_n\}$ of $V_1$ such that for all $i$, $\alpha_i=e_i+\alpha_i',$ where $\alpha_i'\in K\langle e_{n+1},\dots,e_m \rangle$. Similarly, we choose a basis $\{\beta\}$ of $V_2$ such that $\beta=f_1+\beta'$, where $\beta'\in K\langle f_2,\dots,f_n \rangle$. Let $A_{V_1}$ and $A_{V_2}$ be the $m\times n$ and $n\times 1$ matrices formed by the columns $\{\alpha_1,\dots,\alpha_n\}$ and $\{\beta\}$ respectively.

\noindent Then 
\[
\begin{array}{rcl}
    \pi^{-1}(\mathcal{S}) & \cong & \mathcal{S} \times (\GL_n(K)\times\GL_1(K)) \\
    (N_1,N_2) & \mapsto & ((\operatorname{im}(N_1),\operatorname{im}(N_2)),([N_1]_{\{1,\dots,n\}},[N_2]_{\{1\}})) \\   
    (A_{V_1}M_1,A_{V_2}M_2) & \mapsfrom & ((V_1,V_2),(M_1,M_2)),
\end{array}
\]
where $[N]_{\mathcal{I}}$ is the submatrix of $N$ formed by the rows indexed by elements of $\mathcal{I}$. Similar computations apply to the open sets in $\Gr(n,m)\times\Gr(1,n)$ formed by the non-vanishing of any other $n$-minor and $1$-minor. These open sets clearly cover $\Gr(n,m)\times\Gr(1,n)$.
\end{proof}

In order to establish the asserted lower bound on the arithmetic rank, we will need the following vanishing theorems. These results follow from affine vanishing and the Mayer--Vietoris sequence for singular and \'etale cohomologies \cite[III.2.24]{Milne}, 

\begin{theorem}\label{thm:vanishing}
\begin{enumerate}[\quad\rm(1)]
    \item If $X$ is a smooth complex variety of algebraic dimension $d$ that admits an open cover by $t$ affines, then 
\[ H^i_{\mathrm{sing}}(X,\mathbb{Q}) = 0 \quad \text{for all} \ \ i > d+t-1. \] 

    \item If $X$ is a smooth variety of algebraic dimension $d$, over an algebraically closed field $K$, that admits an open cover by $t$ affines, then for any $q$ invertible in $K$
    \[
H^i_{\mathrm{\acute{e}t}}(X,\ZZ/q\ZZ)=0 \quad \text{for all} \ \ i>d + t -1.
\]
\end{enumerate}
\end{theorem}

While \cite{BrunsSchwanzl} determines the arithmetic rank of a determinantal ideal, for our computation, we only need the upper bound, which can be produced more directly as follows:

\begin{lemma}\label{lem:easy}
We have $\ara(I_n(X))\leq mn-n^2+1$ and $\ara(X\underline{y})\leq m$.   
\end{lemma}

\begin{proof}
The second assertion is obvious. For the first assertion, we have the inclusion of the Pl\"{u}cker ring 
\[K[\,\text{all $n \times n$ minors of } X\,] \into K[X].
 \]
 As in the proof of Proposition \ref{thm:nullcone}, a homogeneous system of parameters of the Pl\"{u}cker ring generates its homogeneous maximal ideal up to radical. Viewing this in $K[X]$, we get $\ara(I_n(X))$ is at most the dimension of the Pl\"{u}cker ring, which is $mn-n^2+1$.
\end{proof}

\begin{theorem}\label{thm:singular}
    Let $K=\CC$ and $d=mn+n$, the algebraic dimension of $U$. Then
    $$H^{d+(mn-n^2+n)}_{\mathrm{sing}}(U,\QQ)=\QQ.$$
\end{theorem}
\begin{proof}
Using Lemma \ref{lem:easy}, the statement of Theorem \ref{thm:vanishing} (1) in our case is 
\begin{equation}\label{eqn:vanishingU}
   H^i_{\text{sing}}(U_1,\mathbb{Q}) = 0 \ \text{for all} \ i > d+m-1, 
\end{equation}
\begin{equation}\label{eqn:vanishingV}
  H^i_{\text{sing}}(U_2,\mathbb{Q}) = 0 \ \text{for all} \ i > d+mn-n^2.
\end{equation}

Since $\Gr(n,m) \times \Gr(1,n)$ is simply connected, the Leray spectral sequence for the fibration in Lemma \ref{lem:LTFB} takes the simple
form
\[ E^{u,v}_2 = H_{\text{sing}}^{u}(\Gr(n,m) \times \Gr(1,n),H_{\text{sing}}^{v}(\GL_n(K) \times \GL_1(K),\mathbb{Q})) \Longrightarrow H_{\text{sing}}^{u+v}(U_{12},\mathbb{Q}).\]

Let $D \colonequals \dim(\Gr(n,m)\times \Gr(1,n))=n(m-n)+(n-1)=mn-n^2+n-1$ and note that $\dim(\GL_n(K)\times \GL_1(K))=n^2+1.$ Since 
\[H^{2D}_{\text{sing}}(\Gr(n,m) \times \Gr(1,n),\mathbb{Q}) \cong \mathbb{Q} \quad \text{and}\quad H^{n^2 + 1}_{\text{sing}}(\GL_n(K) \times \GL_1(K),\mathbb{Q}) \cong \mathbb{Q},\]
and the higher cohomology groups vanish, we get that 
\begin{align*}
     H^{2D + n^2 + 1}_{\text{sing}}(U_{12},\mathbb{Q}) \cong E^{2D,n^2 +1}_{\infty} &= E^{2D,n^2 +1}_{2} \\ &\cong H^{2D}_{\text{sing}}(\Gr(n,m) \times \Gr(1,n),H^{n^2 + 1}_{\text{sing}}(\GL_n(K) \times \GL_1(K),\mathbb{Q}))\\
     &\cong H^{2D}_{\text{sing}}(\Gr(n,m) \times \Gr(1,n),\mathbb{Q})\\ &\cong \mathbb{Q}.
\end{align*} 

Let $N \colonequals d + (mn + n - n^2)$ then $2D + n^2 + 1 = d+(mn+n-n^2-1) =N-1$. We have \[ H^{i}_{\text{sing}}(U_{12},\mathbb{Q}) \neq 0 \ \text{for} \ i = N-1. \] Since $m>n>1$, we have $N-1>d+m-1$ and $N-1>d+(mn-n^2+1)$. Hence, by Equations \ref{eqn:vanishingU} and \ref{eqn:vanishingV}, we get \[ H^{i}_{\text{sing}}(U_1,\mathbb{Q}) = 0 \ \text{and} \ H^{i}_{\text{sing}}(U_2,\mathbb{Q}) = 0 \ \text{for} \ i \geq N-1. \]

The Mayer--Vietoris sequence of singular cohomology
\[\to H^{i}_{\text{sing}}(U,\mathbb{Q}) \to H^{i}_{\text{sing}}(U_1,\mathbb{Q}) \oplus H^{i}_{\text{sing}}(U_2,\mathbb{Q}) \to H^{i}_{\text{sing}}(U_{12},\mathbb{Q}) \to H^{i+1}_{\text{sing}}(U,\mathbb{Q}) \to\] 
gives us \[ H^{N-1}_{\text{sing}}(U_{12},\mathbb{Q}) \cong H^{N}_{\text{sing}}(U,\mathbb{Q}) \cong \mathbb{Q}. \qedhere\] 
\end{proof}

\begin{remark}
    The required lower bound for the arithmetic rank of the residual intersection $\RI(m,\underline{y})$ follows immediately from the above calculation. By Theorem \ref{thm:vanishing} (1), we have \[ H^{i}_{\text{sing}}(U,\mathbb{Q}) = 0 \ \text{for all }\ i>d+\ara(\RI(m,\underline{y}))-1. \] 
    By Theorem \ref{thm:singular}, this gives $N\leq d+\ara(\RI(m,\underline{y}))-1$. Therefore 
    \[\ara(\RI(m,\underline{y}))\geq n(m-n+1)+1\] 
\end{remark}

The above calculation works analogously in positive characteristic: 

\begin{theorem}\label{thm:etale}
    Let $K$ be an algebraically closed field of characteristic $p>0$ and let $m$, $n$, $U$ and $d$ be defined as above. Then \[H^{d+(mn-n^2+n)}_{\mathrm{\acute{e}t}}(U,\ZZ/q\ZZ)=\ZZ/q\ZZ.,\] where $q$ is a prime integer other than $p$. 
\end{theorem}
\begin{proof}
    The Grassmann variety is simply connected in the \'etale topology as well. So the \'etale fundamental group  
    \[\pi_{1,\text{\'et}}(\Gr(n,m)\times\Gr(1,n)) = \pi_{1,\text{\'et}}(\Gr(n,m))\times \pi_{1,\text{\'et}}(\Gr(1,n))=0.\] Given this fact, the proof of Theorem \ref{thm:singular} goes through verbatim.
\end{proof}

We are now ready to prove the Main Theorem of this paper.

\begin{theorem}\label{thm:arageneric}
    Let $m$ and $n$ be positive integers with $m\geq n$ and $K$ is a field or the integers. The arithmetic rank of the ideal $\RI(m,\underline{y})$ in $K[X,\underline{y}]$ is 
    \[
\ara(\RI(m,\underline{y})) =\begin{cases} 
n(m-n+1)+1 &   n \ge 2, \\
m & n=1.
\end{cases}
\]
\end{theorem}
\begin{proof}
    The case $n=1$ is immediate since $\RI(m,y_1)$ is an ideal of $m$ variables; we assume $n>1$ from now on. If $K$ is the integers, the assertion is noted in Remark \ref{rmk:Z}. If $K$ is a field of characteristic zero, the result follows from Proposition \ref{thm:nullcone} and Theorem \ref{thm:ASLdimension}. If $K$ has characteristic $p>0$ and $m>n$, the result follows from Theorems \ref{thm:ASLdimension} and \ref{thm:etale}. Note that a lower bound for the arithmetic rank in $\overline{K}[X,\underline{y}]$ is also a lower bound in $K[X,\underline{y}]$, where $\overline{K}$ is the algebraic closure of $K$. It remains to prove the assertion over fields of characteristic $p>0$ when $m$ equals $n$. We work with this assumption for the remainder of the proof. 

    By Theorem \ref{thm:ASLdimension}, we already have that $\ara(RI(m,\underline{y}))\leq n(m-n+1)+1=m+1.$ To prove a lower bound for the arithmetic rank, we may assume that $K$ is algebraically closed. We first consider the case $m=2$. Observe that 
    $$\RI(2,(y_1,y_2))=I_2 \begin{pmatrix}
        x_{11} & x_{21} & -y_2\\
        x_{12} & x_{22} & y_1\\
    \end{pmatrix}$$ is a determinantal ideal. Let $q$ be a prime other than $p$ and let $U(2)$ be the affine open set $K^6\setminus V(\RI(2,\underline{y}))$. Bruns and Schw\"anzl \cite{BrunsSchwanzl} proved that the arithmetic rank of $\RI(2,(y_1,y_2))$ is $3$ by showing that $$H^8_{\acute{e}t}(U(2),\mathbb{Z}/q\mathbb{Z})\neq 0.$$ This settles the case $m=2$. Since the arithmetic rank is exactly 3, it is also known that all higher \'etale cohomologies vanish by Theorem \ref{thm:vanishing} (2). 

    Next, suppose that $m>2$. Specialize the matrix $X$ to  
    \[X' \colonequals\begin{bmatrix}
        x_{11} & x_{12} & 0 & \cdots & 0\\
        x_{21} & x_{22} & 0 & \cdots & 0\\
        0 & 0 & 1 &\cdots & 0\\
        \vdots & \vdots & \vdots & \ddots & \vdots\\
        0 & 0 & 0 & \cdots & 1\\
    \end{bmatrix}.
    \] Since arithmetic rank can only decrease on specialization, it suffices to show that the ideal $J$ obtained from $\RI(m,\underline{y})$ after the above specialization has arithmetic rank bounded below by $n(m-n+1)+1=m+1.$ Observe that 
    \[J=\RI(2,(y_1,y_2))+(y_3,y_4,\dots,y_m).\] 
    Let $U(m)\colonequals K^{m+4}\setminus V(J)$ and note that $U(m)=U_1\cup U_2$, where $$U_1\colonequals K^{m+4}\setminus V(\RI(2,(y_1,y_2)))\cong U(2)\times K^{m-2} \quad \text{and}$$ $$U_2\colonequals K^{m+4}\setminus V(y_3,\dots,y_m)\cong K^6\times (K^*)^{m-2}.$$ It follows that the highest nonvanishing \'etale cohomologies of $U_1$ and $U_2$, with coefficients in $\ZZ/q\ZZ$, are at the indices $8$ and $(m-2)+(m-2)-1=2m-5$ respectively. 
    
    Next, consider their intersection $U_{12}:=U_1\cap U_2\cong U(2)\times (K^*)^{m-2}.$ The highest nonvanishing \'etale cohomology of $U_{12}$ is at the index $8+(2m-5)=2m+3.$ By the Mayer--Vietoris sequence on \'etale cohomology, we get $$H_{\acute{e}t}^{2m+4}(U(m),\ZZ/q\ZZ)\neq 0.$$ Therefore, by Theorem \ref{thm:vanishing} (2), we get $\ara(J)\geq m+1$, as required.
\end{proof}


\begin{proposition}[{\cite{HartshorneFlat}}] \label{thm:Hartshorneflat}
    Let $A$ be a Noetherian local ring containing a field $K$ and let $f_1,\dots,f_n$ be a regular sequence in $A$. Then the natural map of $K$-algebras $$\phi:K[X_1,\dots,X_n]\to A,$$ which sends $X_i$ to $f_i$ for each $i$, is injective, and $A$ is flat as a $K[X_1,\dots,X_n]$-module.
\end{proposition}

We now state the main consequence of Theorem \ref{thm:arageneric}.

\begin{theorem}\label{thm:main}
 Let $(R,\frakm)$ be a Noetherian local (or $\NN$-graded) ring. Let $\underline{f}\colonequals f_1,\ldots ,f_n$ be a regular sequence in $R$ (homogeneous if $R$ is graded) and $I\colonequals (\underline{f})$. Let $\mathfrak{a} \subsetneq I$ be an ideal generated by $m$ elements with $m \geq n$ and $J\colonequals \mathfrak{a}:I$ is an $m$-residual intersection of the ideal $I$. Then the arithmetic rank
\[
\mathrm{ara}(J) \le 
\begin{cases} 
n(m-n+1)+1 &   n \ge 2, \\
m & n=1.
\end{cases}
\]

 The upper bound is attained when each of the following conditions are met:
 \begin{enumerate}[\quad\rm(1)]
     \item The residual intersection $J$ is generic.
     \item $R$ has characteristic $0$.
     \item $R$ contains a field or $\frakm \cap\ZZ=(p)$, with $p$ a nonzerodivisor on $R$. 
 \end{enumerate}    
\end{theorem}

\begin{proof}
By \cite[Example 3.4]{HunekeUlrichRI}, any $m$-residual intersection of $I=(\underline{f})$ is a specialization of the generic $m$-residual intersection $\RI(m,\underline{f})$ of $I$ and thus, the arithmetic rank of $J$ is bounded above by $\ara(RI(m,\underline{f}))$. The case $n=1$ follows from the fact that $\RI(m,f_1)$ is an ideal of $m$ variables. For $n>1$, the arithmetic rank of $\RI(m,\underline{f})$ is bounded above by $n(m-n+1)+1$ due to Theorem \ref{thm:arageneric}.

Now we assume that conditions $(1)$, $(2)$, and $(3)$ are met. The case $n=1$ is immediate. For $n>1$, the inequality $\ara(J)\leq n(m-n+1)+1$ follows from Theorem \ref{thm:ASLdimension}. If $R$ contains a field $K$, the natural map of $K[X]$-algebras $$\phi: K[X][y_1,\dots,y_n]\to R[X],$$ which maps $y_i$ to $f_i$ is flat by Proposition \ref{thm:Hartshorneflat}. On localizing, we get a faithfully flat map $$\phi':K[X][y_1,\dots,y_n]_{(X,y_1,\dots,y_n)}\to R[X]_{(\mathfrak{m},X)}$$ which maps $y_i$ to $f_i.$ Let $S=K[X][y_1,\dots,y_n]_{(X,y_1,\dots,y_n)}$ and $d=n(m-n+1)+1$. Note that $H_{\RI(m,\underline{y})}^d(S)\neq 0$, as shown in Proposition \ref{thm:nullcone}. Let $R'=R[X]_{(\mathfrak{m},X)}$ and $J'=J_{(\mathfrak{m},X)}$. We get that the local cohomology module \[H_{J'}^d(R')=H_{\RI(m,\underline{y})}^d(S)\otimes_S R'\neq 0,\] as $R'$ is a faithfully flat $S$-module. Therefore $\ara(J)\geq\ara(J')\geq d=n(m-n+1)+1.$

In the case that $R$ does not contain a field, but $m\cap \mathbb{Z}=(p)$, with $p$ a nonzerodivisor on $R$, we construct an analogue of the above argument. Under these conditions, the map $$\phi:\ZZ_{(p)}[X][y_1,\dots,y_n]\to R[X]$$ is flat. As $\ZZ_{p}$ is a flat $\ZZ$-module, we repeat the above argument using the crucial input that due to Remark \ref{rmk:Z}, we have
\[H^d_{\RI(m,\underline{y})}(\ZZ[X][y_1,\dots,y_n])\neq 0. \qedhere\]
\end{proof} 

\appendix \section{A second take on the dimension of the auxiliary ring}\label{section:Appendix}

Recall that the dimension of the auxiliary ring $K[B]$ is the arithmetic rank of the residual intersection ideal $\RI(m,\underline{y})$. This dimension was computed in Section \ref{section:ASL} by endowing $K[B]$ with the structure of an ASL. 

In this section, we discuss another strategy to compute the dimension of $K[B]$. We find an explicit transcendence basis for its field of fractions. While the obvious advantage is that this route avoids the combinatorial background needed for the arguments of Section \ref{section:ASL}, the proof---though completely self-contained---turns out to be perhaps more involved. 
\begin{ntn}\label{trans_notation}

\begin{itemize}
\
\item $m \geq n$ are positive integers. $X$ and $\underline{y}$ are $m\times n$ and $n\times 1$ matrices of indeterminates over a field $K$. 

\item $[i_1,\dots,i_n]$ is the maximal minor of $X$ indexed by the rows $i_1<i_2<\cdots< i_n$ of $X$.

\item $Q_i$ is the $i$-th entry of the matrix $X\underline{y}$.

\item The specialization $X'$ of $X$ is  
\[
X'_{ij}= \begin{cases} 0 &  \text{$j\neq 1$, $j\neq i$, and $i\leq n$}, \\ X_{ij} & \text{otherwise}.\end{cases}
\]

\noindent That is,
\[
X' = \begin{bmatrix}
x_{11} & 0 & 0 & \cdots & 0 & 0 \\
x_{21} & x_{22} & 0 & \cdots & 0 & 0 \\
x_{31} & 0 & x_{33} & \cdots & 0 & 0 \\
\vdots & \vdots & \vdots & \ddots & \vdots & \vdots \\
x_{n-1,1} & 0 & 0 & \cdots & x_{n-1,n-1} & 0 \\ 
x_{n,1} & 0 & 0 & \cdots & 0 & x_{n,n} \\
x_{n+1,1} & x_{n+1, 2} & x_{n+1, 3} & \cdots & x_{n+1, n-1} & x_{n+1, n}\\
\vdots & \vdots & \vdots &  & \vdots & \vdots \\
x_{m,1} & x_{m,2} & x_{m,3} & \cdots & x_{m, n-1} & x_{m, n}\\
\end{bmatrix}.
\]

\item The specialization $\underline{{y'}}$ of $\underline{y}$ is 
\[
\underline{y'} = \begin{bmatrix} 1 & 0 & \ldots & 0 & 0
\end{bmatrix}^{T}.
\]

\item $B \colonequals  \{Q_1,\dots,Q_m\}\bigcup \{[i_1,\dots,i_n]~|~ i_1<i_2<\cdots< i_n\}$ and $K[B]$ is the $K$-subalgebra of $K[X,\underline{y}]$ generated by $B$. 

\item For integers $i$ and $j$ with $n < i \leq m$ and $2\leq j \leq n$, set

\[
M_{i,j} = [1,2,\dots,j-1,\widehat{j}, j+1, \dots, n, i].
\]

\noindent In addition, set

\[
M_{n,n} = [1, 2, \dots, n-1, n ].
\]

\end{itemize}
\end{ntn}

\begin{theorem}\label{trans_basis}
Fix Notation \ref{trans_notation}. The set of elements 
\[
D \colonequals \{\,M_{i,j} : n < i \le m,\; 2 \le j \le n\,\} \;\bigcup\; \{M_{n,n},Q_1,\dots,Q_m\}
\]

\noindent is a transcendence basis for the field of fractions of $K[B]$. Therefore
\[\dim(K[B])=n(m-n+1)+1.\]
\end{theorem}

\begin{proof}

We first show that $D$ is an algebraically independent set over $K$. It suffices to find specializations $X'$ and $Y'$ of $X$ and $Y$ respectively such that the corresponding specialization $D'$ of $D$ is algebraically independent. These matrices $X'$ and $Y'$ are as in Notation \ref{trans_notation}.

We begin with describing the elements of $D'$. After specializing $X$ to $X'$ and $Y$ to $Y'$, straightforward calculations show that each $Q_i$ specializes to 
\[
Q_i' = x_{i,1},
\]

\noindent the maximal minor $M_{n,n}$ specializes to 
\[
M_{n,n}' = \prod_{k=1}^n x_{k,k},
\]

\noindent and any other maximal minor $M_{i,j}\in D$ specializes to 
\[
M_{i,j}' =  (-1)^{n+j}x_{i,j}\prod_{k\neq j}^nx_{k,k}.
\]
Suppose that
\begin{equation}\label{alg_relation}
\sum_{\alpha}k_{\alpha}m_{\alpha} = 0
\end{equation}

\noindent is an algebraic relation among the elements of $D'$ over $K$ i.e., each $m_{\alpha}$ is a monomial in the elements of $D'$. For fixed  $m_\alpha$ and $d\in D'$, we set
\[
{\mathrm{exp}(m_{\alpha},d)} \colonequals \max\left\{k\in \mathbb{Z}_{\geq 0}~|~\text{$d^k$ divides $m_\alpha$}\right\}.
\]

\noindent Similarly, for fixed $m_\alpha$ and $z\in X'$, we set
\[
\mathrm{exp}(m_{\alpha},z) \colonequals \max\left\{k\in \mathbb{Z}_{\geq 0}~|~\text{$z^k$ divides $m_\alpha$}\right\}.
\]
\noindent Straightforward calculations show

\begin{equation}\label{x_eq_1}
\exp(m_\alpha,x_{(i,1)}) = \begin{cases} \exp(m_\alpha, Q_1') + \sum\limits_{\substack{2\leq t \leq n \\ n< s \leq m }}  \exp(m_\alpha, M_{s,t}')+ \exp(m_\alpha,M_{n,n}') & i=1
\\
\exp(m_\alpha, Q_i') & 2\leq i \leq m,
\end{cases}
\end{equation}

\begin{equation}\label{x_eq_2}
\exp(m_\alpha, x_{j,j}) = \sum\limits_{\substack{t\neq j \\ n< s \leq m}} \exp(m_\alpha, M_{s,t}') +\exp(m_\alpha, M_{n,n}')
\end{equation}

\noindent  for $2 \leq j \leq n$, and  
\begin{equation}\label{x_eq_3}
\exp(m_\alpha, x_{i,j}) = \exp(m_\alpha, M_{i,j}'),
\end{equation}

\noindent for all $n< i \leq m$ and $2\leq j \leq n$.

From these calculations, we observe that if $\alpha\neq \beta$ then $m_\alpha \neq m_\beta$. Indeed, for sake of contradiction assume that $m_\alpha = m_\beta$. In particular

\[
\exp(m_\alpha,z) = \exp(m_\beta,z)
\]

\noindent for all $z\in X'$. Equation \ref{x_eq_3} implies that 

\begin{equation}\label{M_eq}
\exp(m_\alpha, M_{i,j}') = \exp(m_\beta, M_{i,j}')
\end{equation}
\noindent for all $n< i \leq m$ and $2\leq j \leq n$. But then Equations \ref{x_eq_2} and \ref{M_eq} yield 

\begin{equation}\label{Mn_eq}
\exp(m_\alpha,M_{n,n}')=\exp(m_\beta,M_{n,n}').
\end{equation}

\noindent In turn, Equations \ref{x_eq_1}, \ref{M_eq}, and \ref{Mn_eq} give
\begin{equation}\label{Q_eq}
\exp(m_\alpha,Q_i') = \exp(m_\beta,Q'_i)
\end{equation}
\noindent for all $1\leq i\leq m$. It follows that $\alpha =\beta$. In view of this, the equality $\sum_{\alpha}k_{\alpha}m_{\alpha} = 0$ implies that $k_\alpha = 0$ for all $\alpha$. Hence $D'$ is algebraically independent over $K$.





Now, in order to show $D$ is a transcendence basis of $\mathrm{Frac}(K[B])$, it is enough to show that each element of $B$ is in $\mathrm{Frac}(K[D])$. To this end, set
\[
N_k \colonequals \{[1,j_2,\dots,j_{n-k},i_1,\dots,i_k]~|~ 2 \leq j_2 < j_3 < \cdots < j_{n-k} \leq n < i_1 < \cdots < i_k \leq m\}
\]

\noindent for $k=1, \dots n-1$. We show each $N_k$ is contained in $\mathrm{Frac}(K[D])$ by induction on $k$. This is clear for $n=1$ as $N_1 \subset D$. Inductively, assume that $N_{k-1} \subset \mathrm{Frac}(K[D])$ for each $k$ with $1 \leq k-1 < n-1$; let $[1,j_2,\dots, j_{n-k},i_1,\dots,i_k]\in N_k$. The Pl{\"u}cker relation corresponding to
\[
\{1,j_2,\dots, j_{n-k},i_1,\dots,i_{k-1}\} \quad \text{and \quad} \{1,2,\dots, n-2,n-1,n,i_k\}
\]
\noindent is
\begin{equation}\label{plukrel}
\sum_{s = 1}^{n} P_s +[1,j_2,\dots, j_{n-k},i_1,\dots, i_k][1,2,\dots, n-1,n] = 0,
\end{equation}
where for $s=1,\dots, n$ (before reordering indices).
\[
P_s = [1, j_2, \dots, j_{n-k}, i_1,\dots i_{k-1}, s][1,2 \dots, \widehat{s}, \dots , n, i_k].
\]

Note that if $s \in \{1,j_2,\dots, j_{n-k}\}$, we have $P_s=0$ by convention. In any case $P_s$ is a product of an element of $N_{k-1}$ and an element of $N_1$, and therefore $P_s\in \mathrm{Frac}(K[D])$ for $s=1,\dots n$. Therefore, by Equation \ref{plukrel} and the fact that $[1,2,\dots,n-1, n]\in D$, it follows that $[1,j_2,\dots, j_{n-k},i_1,\dots, i_k]\in \mathrm{Frac}(K[D])$. This completes the induction step. Therefore, $N_k\subset \mathrm{Frac}(K[D])$ for all $k=1,\dots, n-1$. 

We now show that every maximal minor is in $\mathrm{Frac}(K[D])$. Since the union of sets $$\bigcup_{k=1}^{n-1}N_k$$ consists of all maximal minors of X that contain the first row, it suffices to show that every maximal minor that does \textit{not} contain the first row is in $\mathrm{Frac}(K[D])$. To this end, consider a maximal minor $[j_1,\dots, j_n]$ with $2\leq j_1 <\cdots < j_n \leq m$. The matrix 
\[
\begin{bmatrix}
x_{1,1} & x_{1,2} & \cdots & x_{1,n} & Q_{1}\\
x_{j_1,1} & x_{j_1,2} & \cdots & x_{j_1,n} & Q_{j_1} \\
x_{j_2,1} & x_{j_2,2} & \cdots & x_{j_2,n} & Q_{j_2} \\
\vdots & \vdots & \ddots & \vdots & \vdots \\
x_{j_n,1} & x_{j_n,2} & \cdots & x_{j_n,n} & Q_{j_n} \\
\end{bmatrix}
\]
is singular as the last column is a $K[Y]$-linear combination of the other columns. Expanding its determinant along the last column, we obtain
\[
(-1)^{n+2}[j_1,\dots j_n]Q_1 + \sum_{k=1}^n (-1)^{k+n+2}[1,j_1,\dots \widehat{j_k},\dots, j_n]Q_{j_k} = 0.
\]
As each $Q_i\in D$ by design and $[1,j_1,\dots \widehat{j_k},\dots, j_n] \in \mathrm{Frac}(K[D])$, we deduce that $[j_1,\dots, j_n]$ is in the field of fractions $\mathrm{Frac}(K[D])$. This finishes the proof.
\end{proof}

\begin{remark}
We warn the reader that the transcendence basis of $K[B]$ does not give a system of parameters of $K[B]$, or set-theoretic generators of $\RI(X,\underline{y})$. Indeed, in the setup of Example \ref{example:posetorder}, Theorem \ref{trans_basis} asserts that a transcendence basis of $K[B]$ is given by $\{Q_1,Q_2,Q_3,Q_4,[1,2],[1,3],[1,4]\}$. However, 
    \[[3,4]\notin\sqrt{(Q_1,Q_2,Q_3,Q_4,[1,2],[1,3],[1,4])} K[X,\underline{y}].\]

\end{remark}

\section*{Acknowledgments}
We thank Corrado De Concini and Claudio Procesi for their detailed comments on Section \ref{section:ASL} of the paper. We thank Anurag Singh for suggesting an improvement in one of our results. We are also grateful to Barbara Betti, Aldo Conca, Eloísa Grifo, Jack Jeffries, Linquan Ma, Aryaman Maithani, Anurag Singh, Bernd Ulrich, and Uli Walther for several valuable discussions. We thank the referee for their comments. 

MB was supported by NSF Grants DMS-2302430, DMS-2100288, and Simons Foundation Grant SFI-MPS-TSM-00012928. KMS was supported by NSF Grant DMS-2236983. TM was supported by NSF Grant DMS-2044833. 

\bibliographystyle{alpha}
\bibliography{main.bib}

\end{document}